\documentclass{amsart}

\newtheorem{theorem}{Theorem}[section]
\newtheorem{lemma}[theorem]{Lemma}

\newtheorem{corollary}[theorem]{Corollary}

\theoremstyle{definition}
\newtheorem{definition}[theorem]{Definition}

\usepackage{amscd,amssymb}
\begin{document}

\title[The Nowicki Conjecture
for free metabelian Lie algebras]
{The Nowicki Conjecture\\
for free metabelian Lie algebras}
\author[Vesselin Drensky, {\c S}ehmus F{\i}nd{\i}k]
{Vesselin Drensky and {\c S}ehmus F{\i}nd{\i}k}
\address{Institute of Mathematics and Informatics,
Bulgarian Academy of Sciences,
1113 Sofia, Bulgaria}
\email{drensky@math.bas.bg}
\address{Department of Mathematics,
\c{C}ukurova University, 01330 Balcal\i,
 Adana, Turkey}
\email{sfindik@cu.edu.tr}

\thanks
{The research of both authors was partially supported by Grant I02/18 of the Bulgarian Science Fund.}
\thanks
{The research of the second named author was partially supported by the
Council of Higher Education (Y\"OK) in Turkey.}


\begin{abstract}
Let $K[X_d]=K[x_1,\ldots,x_d]$ be the polynomial algebra in $d$ variables over a field $K$ of characteristic 0.
The classical theorem of Weitzenb\"ock from 1932 states that for linear locally nilpotent derivations $\delta$ (known as Weitzenb\"ock derivations)
the algebra of constants $K[X_{d}]^{\delta}$ is finitely generated.
When the Weitzenb\"ock derivation $\delta$ acts on the polynomial algebra $K[X_d,Y_d]$ in $2d$ variables by
$\delta(y_i)=x_i$, $\delta(x_i)=0$, $i=1,\ldots,d$,
Nowicki conjectured that $K[X_d,Y_d]^{\delta}$ is generated by $X_d$ and $x_iy_j-y_ix_j$ for all $1\leq i<j\leq d$.
There are several proofs based on different ideas confirming this conjecture.
Considering arbitrary Weitzenb\"ock derivations of the free $d$-generated metabelian Lie algebra $F_d$,
with few trivial exceptions, the algebra $F_d^{\delta}$ is not finitely generated.
However, the vector subspace $(F_d')^{\delta}$ of the commutator ideal $F_d'$ of $F_d$ is finitely generated
as a $K[X_d]^{\delta}$-module. In this paper we study an analogue of the Nowicki conjecture in the Lie algebra setting and
give an explicit set of generators of the $K[X_d,Y_d]^{\delta}$-module $(F_{2d}')^{\delta}$.
\end{abstract}

\maketitle

\section{Introduction}

A linear operator $\delta$ of a (not necessarily commutative or associative) algebra $R$ over a field $K$ is called a {\it derivation} if
\[
\delta(uv)=\delta(u)v+u\delta(v)\text{ for all }u,v\in R.
\]
The kernel $R^{\delta}$ of $\delta$ is called the {\it algebra of constants of} $\delta$.

In the sequel $K$ will be a field of characteristic 0.  Let  $K[X_d]=K[x_1,\ldots,x_d]$,  $d\geq 2$, be the polynomial algebra
in $d$ variables over $K$. A derivation  $\delta$ of  $K[X_d]$ acting as a nonzero nilpotent linear operator of the vector space
$KX_d$ with basis $X_d$ is called a {\it Weitzenb\"ock derivation}. The Jordan normal form $J(\delta)=(J_1,\ldots ,J_s)$ of the matrix of $\delta$
considered as a linear operator acting on $KX_d$ consists of Jordan cells $J_i$, $i=1,\ldots ,s$, with zero diagonals.
In 1932 Weitzenb\"ock \cite{W} proved that the algebra of constants
\[
K[X_d]^{\delta}=\ker{\delta}=\{u\in K[X_d]\mid \delta(u)=0\}
\]
is finitely generated. For more information on Weitzenb\"ock derivations
one can see the books by Nowicki \cite{N},  Derksen and Kemper \cite{DK}, and Sturmfels \cite{St}.
The algebra of constants $K[X_d]^{\delta}$ can be considered also from the point of view of classical invariant theory.
The linear operator $\alpha\delta$ of $KX_d$ is nilpotent
for all $\alpha\in K$. The exponent
\[
\exp(\alpha\delta)=1+\frac{\alpha\delta}{1!}+\frac{\alpha^2\delta^2}{2!}+\cdots
\]
is a well defined invertible linear operator of $KX_d$ and this defines a $d$-dimensional representation of the unitriangular group
\[
UT_2(K)=\left\{\left(\begin{matrix}1&\alpha\\
0&1\\
\end{matrix}\right)\mid\alpha\in K\right\}.
\]
This action can be extended diagonally on the whole algebra $K[X_d]$ and $K[X_d]^{\delta}$ is equal to the algebra of invariants $K[X_d]^{UT_2(K)}$.

Let $K[X_d,Y_d]=K[x_1,\ldots,x_d,y_1,\ldots,y_d]$, be the polynomial algebra
in $2d$ variables and let ${\delta}$ be the Weitzenb\"ock derivation defined by ${\delta}(y_i)=x_i$, ${\delta}(x_i)=0$, $i=1,\ldots,d$.
In the language of invariant theory $K[X_d,Y_d]^{\delta}$ is the algebra of invariants for the action of the additive group $(K,+)$ on $K[X_d,Y_d]$ by
\[
\alpha:x_i\to x_i,y_i\to y_i+\alpha x_i,\quad i=1,\ldots,d, \quad\alpha\in K.
\]
In 1994 Nowicki conjectured \cite[p. 76, Conjecture 6.9.10]{N} that the algebra of constants  $K[X_d,Y_d]^{\delta}$ is generated by $x_1,\ldots,x_d$
and the determinants
\begin{equation}\label{u_pq}
u_{pq}=\begin{vmatrix}
x_p & y_p \\
x_q &y_q\\
\end{vmatrix}, \quad 1\leq p<q \leq d.
\end{equation}
This conjecture attracted many mathematicians and was verified by several authors with proofs based on different ideas:
In his Ph.D. thesis in 2004 Khoury \cite{K1, K2} gave a computational proof using Gr\"obner basis techniques.
The unpublished proofs of Derksen and Panyushev applied ideas of classical invariant theory.
Several proofs appeared in 2009. Drensky and Makar-Limanov \cite{DML} gave an elementary proof using easy arguments from undergraduate algebra
and a simple induction only, without involving any invariant theory. In his proof Bedratyuk \cite{Bed}
reduced the Nowicki conjecture to a well known problem of classical invariant theory.
Kuroda \cite{Kuroda} gave a short proof based on the ideas of Kurano \cite{Kurano} in his study on the analogue in positive characteristic
of the Roberts' counterexample to the Hilbert fourteenth problem. As Kuroda mentioned Hashimoto informed him that
Goto, Hayasaka, Kurano, and Nakamura \cite{GHKN} and Miyazaki \cite{M} determined sets of generators for certain invariant rings
where $K[X_d,Y_d]^{\delta}$ is included,
and this gives one more proof of the Nowicki conjecture.

Let $K\langle X_d\rangle$, $d\geq 2$, be the free (unitary or nonunitary) associative algebra freely generated by $X_d$
and let, as above, $\delta$ be a nilpotent linear operator
acting on the vector space $KX_d$. Then the action of $\delta$ on $KX_d$ can be extended to an action as a derivation on the whole algebra $K\langle X_d\rangle$.
If $V$ is a T-ideal (or a verbal ideal) of $K\langle X_d\rangle$, i.e., an ideal which is invariant under all endomorphisms of $K\langle X_d\rangle$, then
it is well known that $\delta(V)\subseteq V$ and $\delta$ induces a derivation on the factor algebra $K\langle X_d\rangle/V$.
We shall use the same notation $\delta$ for this derivation of $K\langle X_d\rangle/V$ and again shall call it Weitzenb\"ock.
The factor algebra $F_d({\mathfrak V})=K\langle X_d\rangle/V$ is a {\it relatively free algebra in the variety} $\mathfrak V$ of associative algebras
defined by the polynomial identities from $V$. As in the case of polynomial algebras,
the kernel $F_d({\mathfrak V})^{\delta}$ of $\delta$ is the algebra of constants of $\delta$.
Similarly, if $L_d=L(X_d)$, $d\geq 2$, is the free Lie algebra freely generated by $X_d$ and $W$ is a T-ideal (or a verbal ideal)
of $L_d$, then the action of $\delta$ on $KX_d$ defines a Weitzenb\"ock derivation on the relatively free algebra $F_d({\mathfrak W})=L_d/W$
in the variety of Lie algebras $\mathfrak W$ defined by the polynomial identities from $W$.
See e.g., the book by Bahturin \cite{B} for a background on varieties of Lie algebras,
the book \cite{D} by one of the authors for associative PI-algebras,
and his paper \cite{DG} with Gupta for Weitzenb\"ock derivations
acting on free and relatively free algebras, and for the properties of their algebras of constants.

In the sequel, let
\[
F_d=F_d({\mathfrak A}^2)=L_d/L_d''
\]
be the factor algebra of $L_d$ modulo the second term $L_{2d}''$ of the derived series of $L_d$.
This is the {\it free metabelian Lie algebra} generated by $X_d$. It is a relatively free algebra in the variety ${\mathfrak A}^2$ of
the metabelian (solvable of class 2) Lie algebras defined by the identity
$[[x_1,x_2],[x_3,x_4]] = 0$. The variety ${\mathfrak A}^2$ has a key position in the theory of varieties of Lie algebras.
By the well-known dichotomy a variety $\mathfrak W$ of Lie algebras either satisfies the Engel condition
and by the theorem of Zelmanov \cite{Z} is nilpotent or contains the metabelian variety ${\mathfrak A}^2$.
Since finitely generated nilpotent algebras are finite dimensional, the algebra $F_d=F_d({\mathfrak A}^2)$ is the minimal relatively free algebra
which is not finite dimensional. If $\delta$ is a Weitzenb\"ock derivation of $F_d$, then
Drensky and Gupta \cite{DG} showed that $F_d^{\delta}$ is finitely generated only in the trivial case
when the Jordan normal form of $\delta$ consists of one Jordan cell of size $2\times 2$ and $d-2$ Jordan cells of size $1\times 1$, i.e.,
when the rank of the matrix of $\delta$ is equal to 1. The commutator ideal $F_d'$ has a natural structure of a $K[X_d]$-module.
Recently Dangovski and the authors \cite{DDF} established that the vector space $(F_d')^{\delta}$ of the constants of $\delta$ in the commutator
ideal $F_d'$ of $F_d$ is a finitely generated $K[X_d]^{\delta}$-module.
Freely speaking, this means that the algebra of constants $F_d^{\delta}$ is very close to be finitely generated.

In the present paper we consider the free metabelian Lie algebra $F_{2d}$  of rank $2d$ generated by the set $X_d\cup Y_d$.
We assume that ${\delta}$ is its Weitzenb\"ock derivation acting similarly as in the Nowicki conjecture.
We give a complete set of generators of the $K[X_d,Y_d]^{\delta}$-module $(F_{2d}')^{\delta}$.
This gives also an infinite set of generators of the Lie algebra $(F_{2d})^{\delta}$.

\section{Preliminaries}

Till the end of the paper we fix the notation $F_{2d}=L_{2d}/L_{2d}''$ for the free metabelian Lie algebra of rank $2d$
freely generated by $X_d\cup Y_d=\{x_1,\ldots,x_d,y_1,\ldots,y_d\}$. We assume that all Lie commutators are left normed, e.g.,
\[
[z_1,z_2,z_3]=[[z_1,z_2],z_3]=[z_1,z_2]\text{ad}z_3
\]
for all $z_1,z_2,z_3\in F_{2d}$. The metabelian identity implies, see, e.g., \cite{B}, that
\[
[z_{j_1},z_{j_2},z_{j_{\sigma(3)}},\ldots,z_{j_{\sigma(k)}}]
=[z_{j_1},z_{j_2},z_{j_3},\ldots,z_{j_k}],
\]
where $\sigma$ is an arbitrary permutation of $3,\ldots,k$. Thus the polynomial algebra $K[X_d,Y_d]$ acts on $F_{2d}'$ by the rule
\[
uf(x_1,\ldots,x_d,y_1,\ldots,y_d)=uf(\text{ad}x_1,\ldots,\text{ad}x_d,\text{ad}y_1,\ldots,\text{ad}y_d),
\]
where $u\in F_{2d}'$, $f(X_d,Y_d)=f(x_1,\ldots,x_d,y_1,\ldots,y_d)\in K[X_d,Y_d]$.

We construct the abelian wreath product due to Shmel'kin \cite{Sh}.
Let $\Lambda_{2d}=K(A_d\cup B_d)$ and $\Gamma_{2d}=K(P_d\cup Q_d)$ denote the abelian Lie algebras with linear bases
\[
A_d\cup B_d=\{a_1,\ldots,a_d\}\cup\{b_1,\ldots,b_d\}\text{ and }P_d\cup Q_d=\{p_1,\ldots,p_d\}\cup\{q_1,\ldots,q_d\},
\]
respectively,
and let $C_{2d}$ be the free right
$K[X_d,Y_d]$-module with free generators $A_d\cup B_d$.
Equipping $C_{2d}$ with trivial multiplication we give it the structure of an abelian Lie algebra.
The abelian wreath product
$W_{2d}=\Lambda_{2d}\text{\rm wr}\Gamma_{2d}$ is equal to the semidirect sum $C_{2d}\leftthreetimes \Gamma_{2d}$. The elements
of $W_{2d}$ are of the form
\[
\sum_{i=1}^da_if_i(X_d,Y_d)+\sum_{i=1}^db_ig_i(X_d,Y_d)+\sum_{i=1}^d\alpha_ip_i+\sum_{i=1}^d\beta_iq_i,
\]
where $\alpha_i,\beta_i\in K$.
The multiplication in $W_{2d}$ is defined by
\[
[C_{2d},C_{2d}]=[\Gamma_{2d},\Gamma_{2d}]=0,
\]
\[
[a_if_i(X_d,Y_d),p_j]=a_if_i(X_d,Y_d)x_j,\quad [b_if_i(X_d,Y_d),p_j]=b_if_i(X_d,Y_d)x_j,
\]
\[
[a_if_i(X_d,Y_d),q_j]=a_if_i(X_d,Y_d)y_j,\quad [b_if_i(X_d,Y_d),q_j]=b_if_i(X_d,Y_d)y_j,
\]
$j=1,\ldots,d$.
Hence $W_{2d}$ is a metabelian Lie algebra and every mapping
\[
\{x_1,\ldots,x_n,y_1,\ldots,y_n\}\to W_{2d}
\]
can be extended to a homomorphism $F_{2d}\to W_{2d}$. As a special case of the embedding theorem of Shmel'kin,
the homomorphism $\varepsilon:F_{2d}\to W_{2d}$ defined by
\[
\varepsilon(x_i)=a_i+p_i, \quad \varepsilon(y_i)=b_i+q_i,\quad i=1,\ldots,d,
\]
is a monomorphism. By this action of $\varepsilon$, the commutator ideal $F_{2d}'$ is embedded into the free right
$K[X_d,Y_d]$-module
\[
C_{2d}=a_1K[X_d,Y_d]\oplus \cdots \oplus a_dK[X_d,Y_d] \oplus b_1K[X_d,Y_d]\oplus \cdots \oplus b_dK[X_d,Y_d]
\]
as follows:
\[
\varepsilon([x_i,x_j])= a_ix_j-a_jx_i, \quad \varepsilon([y_i,y_j])= b_iy_j-b_jy_i, \quad \varepsilon([x_i,y_j])= a_iy_j-b_jx_i,
\]
and then, by induction, if $w\in F_{2d}'$, then
\[
\varepsilon([w,x_j])=\varepsilon(w)x_j,\quad \varepsilon([w,y_j])=\varepsilon(w)y_j,\quad j=1,\ldots,d.
\]
As a consequence of this construction, we have the following result.

\begin{lemma}\label{Lie element}\cite[Theorem 2]{Sh}
An element
\[
\sum_{i=1}^da_if_i(X_d,Y_d)+\sum_{i=1}^db_ig_i(X_d,Y_d)
\]
 from $C_{2d}$ is an image of an element from the commutator ideal $F_{2d}'$ if and only if
\[
\sum_{i=1}^dx_if_i(X_d,Y_d)+\sum_{i=1}^dy_ig_i(X_d,Y_d)=0.
\]
\end{lemma}
It follows immediately from Lemma \ref{Lie element} that the image $\varepsilon(F_{2d}')$ of $F_{2d}'$ in $C_{2d}$
is a submodule of the $K[X_d,Y_d]$-module $C_{2d}$.
In the sequel we shall identify the elements of $F_{2d}'$ with their images in $C_{2d}$.

If $\delta$ is a Weitzenb\"ock derivation of $K[X_d,Y_d]$ we shall assume that it acts on $K[X_d,Y_d]$ by the rule
\[
\delta(y_i)=x_i, \delta(x_i)=0,\quad i=1,\ldots,d,
\]
and shall use without reference that
the algebra of constants $K[X_d,Y_d]^{\delta}$ is generated by $x_1,\ldots,x_d$
and the determinants (\ref{u_pq})
as conjectured by Nowicki \cite{N} and proved in \cite{K1, K2, DML, Bed, Kuroda}.
In this special case the Jordan normal form $J(\delta)$ of $\delta$ consist of $2\times 2$ Jordan cells only, i.e.,
\[
J(\delta)=\left(\begin{matrix}
0&1&\cdots&0&0\\
0&0&\cdots&0&0\\
\vdots&\vdots&\ddots&\vdots&\vdots\\
0&0&\cdots&0&1\\
0&0&\cdots&0&0\\
\end{matrix}\right).
\]
The action of $\delta$
on $\{a_i,b_i\mid i=1,\ldots,d\}$ will be defined in the same way as on $\{x_i,y_i\mid i=1,\ldots,d\}$.
Thus $\delta$ is extended to a derivation of $F_{2d}$ and $W_{2d}$.
The vector space $C_{2d}^{\delta}$ of the constants of $\delta$ in the free $K[X_d,Y_d]$-module $C_{2d}$
is a $K[X_d,Y_d]^{\delta}$-module. The following theorem is a partial case of a result of \cite{DDF}.

\begin{theorem}
Let $\delta$ be a Weitzenb\"ock derivation of the free metabelian Lie algebra $F_{2d}$. Then
the vector space $(F_{2d}')^{\delta}$ of the constants of $\delta$ in the commutator
ideal $F_{2d}'$ of $F_{2d}$ is a finitely generated $K[X_d,Y_d]^{\delta}$-module.
\end{theorem}

Since it will not cause misunderstanding with the notation for the polynomial algebra we shall use the notation
\[
K[X_d,Y_d]^{\delta}=K[X_d,U]=K[X_d,u_{pq}\mid 1\leq p<q\leq d]
\]
for the algebra
generated by $X_d$ and the elements $U=\{u_{pq}\mid 1\leq i<j<k\leq d\}$ defined in (\ref{u_pq}).
Drensky and Makar-Limanov \cite{DML} showed that the algebra $K[X_d,Y_d]^{\delta}$
has the following defining relations
\begin{equation} \label{S1}
x_iu_{jk}-x_ju_{ik}+x_ku_{ij}=0,\quad 1\leq i<j<k\leq d.
\end{equation}
\begin{equation}  \label{R1}
u_{ij}u_{kl}-u_{ik}u_{jl}+u_{il}u_{jk}=0,\quad 1\leq i<j<k<l\leq d.
\end{equation}
and gave a {\it canonical} linear basis consisting of the elements of the form
\begin{equation}  \label{B1}
x_{i_1}\cdots x_{i_m}u_{k_1l_1}\cdots u_{k_sl_s}
\end{equation}
such that the generators $u_{k_{\alpha}l_{\alpha}}$ and $u_{k_{\beta}l_{\beta}}$
do not intersect each other and $u_{k_{\alpha}l_{\alpha}}$ does
not cover  $x_{i_{\gamma}}$ for any $\alpha,\beta,\gamma$.
Here each $u_{k_{\alpha}l_{\alpha}}$ is identified with the open interval $(k_{\alpha},l_{\alpha})$ on the real line. The generators
$u_{k_{\alpha}l_{\alpha}}$ and $u_{k_{\beta}l_{\beta}}$ intersect each other if the intervals $(k_{\alpha},l_{\alpha})$ and $(k_{\beta},l_{\beta})$
have a nonempty intersection and are not contained in each other. We say also that
$u_{k_{\alpha}l_{\alpha}}$ covers  $x_{i_{\gamma}}$ if $i_{\gamma}$ belongs to the open interval  $(k_{\alpha},l_{\alpha})$.
The order of generators in this basis is assumed to be as follows: $p_1 \leq \cdots \leq p_s$ and if $p_n=p_{n+1}$, then $q_n\leq q_{n+1}$;
and the order among $x_{i_{\gamma}}$ is such that $i_1 \leq \cdots \leq i_m$.

As a direct consequence of the affirmative answer to the Nowicki Conjecture
the algebra of constants $K[A_d,B_d,X_d,Y_d]^{\delta}$ of the derivation $\delta$ acting on the polynomial algebra $K[A_d,B_d,X_d,Y_d]$
is generated by $A_d,X_d$ and the determinants
\[
\begin{vmatrix}
a_p & b_p \\
a_q &b_q\\
\end{vmatrix},\quad
u_{pq}=\begin{vmatrix}
x_p & y_p \\
x_q &y_q\\
\end{vmatrix},\quad 1\leq p<q\leq d,
\]
\begin{equation}\label{w_pq}
w_{pq}=a_py_q-b_px_q=\begin{vmatrix}
a_p & b_p \\
x_q &y_q\\
\end{vmatrix},\quad p,q=1,\ldots,d.
\end{equation}
Hence the $K[X_d,U]$-module $C_{2d}^{\delta}$ is generated by the elements $a_1,\dots,a_d$
and the determinants (\ref{w_pq}),
and as a vector space $C_{2d}^{\delta}$ is spanned by the elements of the form
\begin{equation} \label{span1}
a_{i_0}x_{i_1}\cdots x_{i_m}u_{k_1l_1}\cdots u_{k_sl_s}
\end{equation}
\begin{equation} \label{span2}
w_{p_0q_0}x_{j_1}\cdots x_{j_n}u_{p_1q_1}\cdots u_{p_rq_r}
\end{equation}
for $i_0,p_0,q_0=1,\ldots,d$. Ordering the elements $A_d\cup B_d\cup X_d\cup Y_d$
and assuming that the elements from $A_d$ and $B_d$ precede, respectively, the elements from $X_d$ and $Y_d$,
we obtain as an application of (\ref{S1}) and (\ref{R1})
that the $K[X_d,U]$-module $C_{2d}^{\delta}$ has the following defining relations
\begin{equation} \label{S2}
a_iu_{jk}-w_{ik}x_j+w_{ij}x_k=0,\quad 1\leq i\leq d, \,1\leq j<k\leq d,
\end{equation}
\begin{equation}  \label{R2}
w_{ij}u_{kl}-w_{ik}u_{jl}+w_{il}u_{jk}=0,\quad 1\leq i\leq d, \,1\leq j<k<l\leq d.
\end{equation}
In order to fix a basis of $C_{2d}^{\delta}$ as a vector space,
the factors $x_{i_1}\cdots x_{i_m}u_{k_1l_1}\cdots u_{k_sl_s}$ and $x_{j_1}\cdots x_{j_n}u_{p_1q_1}\cdots u_{p_rq_r}$ of the elements
(\ref{span1}) and (\ref{span2}) have to satisfy the restrictions in (\ref{B1}). Additionally, for the elements in (\ref{span2})
we require $q_0\leq p_1$.

\section{Main Results}
In this section we give the generators of the $K[X_d,Y_d]^{\delta}$-module of constants $(F_{2d}')^{\delta}$ in the commutator ideal $F_{2d}'$
of the free metabelian Lie algebra $F_{2d}$.
Since $(F_{2d}')^{\delta}$ is canonically embedded in $C_{2d}^{\delta}$ we shall work in $C_{2d}^{\delta}$ instead of directly in $(F_{2d}')^{\delta}$.

\begin{definition}
We define the $K[X_d,U]$-submodule $L$ of $C_{2d}^{\delta}$ generated by the elements
\begin{equation} \label{g1}
w_{ii},\quad \quad 1\leq i\leq d,
\end{equation}
\begin{equation} \label{g2}
w_{ij}+w_{ji},\quad 1\leq i<j\leq d,
\end{equation}
\begin{equation}\label{g3}
a_ix_j-a_jx_i,\quad 1\leq i<j\leq d,
\end{equation}
\begin{equation}\label{g4}
a_iu_{pq}-w_{pq}x_i,\, 1\leq i\leq d, \, 1\leq p<q\leq d,
\end{equation}
\begin{equation} \label{g6}
a_iu_{jk}-a_ju_{ik}+a_ku_{ij},\quad 1\leq i<j<k\leq d,
\end{equation}
\begin{equation}\label{g5}
w_{ij}u_{pq}-w_{pq}u_{ij}, \, 1\leq i<j\leq d,  \, 1\leq p<q\leq d.
\end{equation}
\end{definition}
By Lemma \ref{Lie element}, one can easily observe that the generating elements (\ref{g1})--(\ref{g5}) of $L$ are Lie elements,
i.e., images of elements in the commutator ideal $F_{2d}'$ of $F_{2d}$.

\begin{lemma}\label{linear span}
The following elements span the quotient space $C_{2d}^{\delta}/L$.
\begin{equation} \label{span4}
w_{p_0q_0}u_{p_1q_1}\cdots u_{p_rq_r},
\end{equation}
\begin{equation} \label{span3}
a_{i_0}x_{i_1}\cdots x_{i_m}u_{k_1l_1}\cdots u_{k_sl_s},
\end{equation}
where $u_{p_0q_0}u_{p_1q_1}\cdots u_{p_rq_r}$ and $x_{i_0}x_{i_1}\cdots x_{i_m}u_{k_1l_1}\cdots u_{k_sl_s}$
are elements of the form (\ref{B1}); i.e., they are canonical basis elements of the algebra $K[X_d,U]$.
\end{lemma}

\begin{proof}
We shall work in the vector space $C_{2d}^{\delta}$ modulo the subspace $L$.
It is sufficient to handle the basis elements of $C_{2d}^{\delta}$ of the form (\ref{span1}) and (\ref{span2}).
Starting with the element $w_{p_0q_0}x_{j_1}\cdots x_{j_n}u_{p_1q_1}\cdots u_{p_rq_r}$ in (\ref{span2})
we apply the relation $w_{pq}x_i\equiv a_iu_{pq}$ (mod $L$) from (\ref{g4})
and bring the element from (\ref{span2}) to an element from (\ref{span1}). If the element from (\ref{span2})
is of the form $w_{p_0q_0}u_{p_1q_1}\cdots u_{p_rq_r}$, then (\ref{g1}) and (\ref{g2}) imply that we may assume that $p_0<q_0$.
Then using the relation (\ref{R2}) we obtain that
the interval $(p_0,q_0)$ does not intersect with $(p_1,q_1),\ldots,(p_r,q_r)$, and the generator (\ref{g5}) fixes the order among
$u_{p_0q_0},u_{p_1q_1}\ldots, u_{p_rq_r}$. This closes the case (\ref{span2}).
Now we consider the element $a_{i_0}x_{i_1}\cdots x_{i_m}u_{k_1l_1}\cdots u_{k_sl_s}$ in (\ref{span1}).
By assumption, the integers $i_1,\ldots,i_m$ do not belong to the open intervals $(p_l,q_l)$.
If $i_0\in (p_l,q_l)$ for some $l=1,\ldots,s$, then the relation
$a_ju_{ik}\equiv a_iu_{jk}+a_ku_{ij}$ (mod $L$) from (\ref{g6}) replaces $a_{i_0}u_{p_lq_l}$ with
$a_{p_l}u_{i_0q_l}$ and $a_{q_l}u_{p_li_0}$. Since the intervals $(i_0,q_l)$ and $(p_l,i_0)$ are shorter than the interval $(p_l,q_l)$,
the integers $i_1,\ldots,i_m$ are not covered by the intervals $(p_1,q_1),\ldots,(p_s,q_s)$.
In finite number of steps the same holds for the integer $i_0$.
Finally the generator (\ref{g3}) fixes the order among
$x_{i_t}$, $t=0,1,\ldots,m$.
\end{proof}

\begin{theorem}
The $K[X_d,U]$-module $L$ consists of all Lie elements in $C_{2d}^{\delta}$.
\end{theorem}
\begin{proof} Let
\[
\sum\xi_{ikl}a_{i_0}x_{i_1}\cdots x_{i_m}u_{k_1l_1}\cdots u_{k_sl_s}
+\sum\psi_{pq}w_{p_0q_0}u_{p_1q_1}\cdots u_{p_rq_r}
\]
be a Lie element in the vector space $C_{2d}^{\delta}/L$. Then by Lemma \ref{Lie element}
and Lemma \ref{linear span} we have that
\[
\sum\xi_{ikl}x_{i_0}x_{i_1}\cdots x_{i_m}u_{k_1l_1}\cdots u_{k_sl_s}
+\sum\psi_{pq}u_{p_0q_0}u_{p_1q_1}\cdots u_{p_rq_r}=0,
\]
where $x_{i_0}x_{i_1}\cdots x_{i_m}u_{k_1l_1}\cdots u_{k_sl_s}$ and $u_{p_0q_0}u_{p_1q_1}\cdots u_{p_rq_r}$
are basis elements of the algebra $K[X_d,U]$.
Clearly each element from the first sum is linearly independent from the elements of the second sum, since there is
at least one multiplier of the form $x_{i_0}$ in each summand of the first sum which does not appear in the second sum.
This implies that
\[
\sum\xi_{ikl}x_{i_0}x_{i_1}\cdots x_{i_m}u_{k_1l_1}\cdots u_{k_sl_s}=\sum\psi_{pq}u_{p_0q_0}u_{p_1q_1}\cdots u_{p_rq_r}=0.
\]
Thus $\xi_{ikl}=0=\psi_{pq}$ for all $i,k,l,p,q$,
because $x_{i_0}x_{i_1}\cdots x_{i_m}u_{k_1l_1}\cdots u_{k_sl_s}$ and $u_{p_0q_0}u_{p_1q_1}\cdots u_{p_rq_r}$
uniquely determine the monomials
$a_{i_0}x_{i_1}\cdots x_{i_m}u_{k_1l_1}\cdots u_{k_sl_s}$ and $w_{p_0q_0}u_{p_1q_1}\cdots u_{p_rq_r}$, respectively.
\end{proof}

Finally the generators of $(F_{2d}')^{\delta}$ are obtained by computing the inverse images of generators of $L$.
\begin{corollary}
The $K[X_d,U]$-module $(F_{2d}')^{\delta}$ is generated by the following elements
\[
[x_i,y_i],\quad \quad 1\leq i\leq d,
\]
\[
[x_i,x_j],\quad 1\leq i<j\leq d,
\]
\[
[x_i,y_j]+[x_j,y_i],\quad 1\leq i<j\leq d,
\]
\[
[x_i,x_p,y_q]-[x_i,y_p,x_q],\, 1\leq i\leq d, \, 1\leq p<q\leq d,
\]
\[
[x_i,x_j,y_k]-[x_i,x_k,y_j]+[x_j,x_k,y_i],\quad 1\leq i<j<k\leq d,
\]
and
\[
[x_i,x_p,y_j,y_q]+[y_i,y_p,x_j,x_q]-[x_i,y_p,y_j,x_q]-[y_i,x_p,x_j,y_q],
\]
where $1\leq i<j\leq d$, $1\leq p<q\leq d$.
\end{corollary}

\section*{Acknowledgements}

The second named author is very thankful to the Institute of Mathematics and Informatics of
the Bulgarian Academy of Sciences for the creative atmosphere and the warm hospitality during his visit
as a post-doctoral fellow when this project was carried out.

\end{document}